\documentclass[11pt]{article}
\usepackage{amsmath, amsfonts, amssymb, amsthm}
\newtheorem{Theorem}{Theorem}[section]

\newtheorem{Lemma}[Theorem]{Lemma}

\newtheorem{Corollary}[Theorem]{Corollary}

\newtheorem{Open Problem}[Theorem]{Open Problem}

\newcommand{\mysection}[1]{\section{#1}\setcounter{equation}{0}}

\begin{document}

\title{ Existence of ground state solution of Nehari-Poho\v{z}aev type for 
a quasilinear Schr\"{o}dinger system }

\author{Jianqing Chen\ and Qian Zhang\thanks{Corresponding author:\ qzhang\_fjnu@163.com  (Q. Zhang) }\\
\small  \  College of Mathematics and Informatics \& FJKLMAA, Fujian Normal University, \\
\small  Qishan Campus, Fuzhou 350117, P. R. China\\
\small  jqchen@fjnu.edu.cn (J. Chen) \ \   qzhang\_fjnu@163.com  (Q. Zhang) }

\date{}

\maketitle
\noindent {\bf Abstract}:  This paper is concerned with the following quasilinear Schr\"{o}dinger system in the entire space $\mathbb R^{N}$($N\geq3$):
$$\left\{\aligned &-\Delta u+A(x)u-\frac{1}{2}\triangle(u^{2})u=\frac{2\alpha}{\alpha+\beta}|u|^{\alpha-2}u|v|^{\beta},\\
&-\Delta v+Bv-\frac{1}{2}\triangle(v^{2})v=\frac{2\beta}{\alpha+\beta}|u|^{\alpha}|v|^{\beta-2}v.\endaligned\right. $$
By establishing a suitable constraint set and studying related minimization problem, we prove the existence of ground state solution for $\alpha,\beta>1$, $2<\alpha+\beta<\frac{4N}{N-2}$. Our results can be looked on as a generalization to results by Guo and Tang (Ground state solutions for quasilinear Schr\"{o}dinger systems, J. Math. Anal. Appl. 389 (2012) 322).

 \medskip


\mysection {Introduction}
Let $N\geq 3$, $u:=u(x),v:=v(x)$ be real valued functions on $\mathbb R^{N}$. In this paper, we consider the following  quasilinear Schr\"{o}dinger system:
\begin{equation}\label{eq11}
\left\{\aligned &-\Delta u+A(x)u-\frac{1}{2}\triangle(u^{2})u=\frac{2\alpha}{\alpha+\beta}|u|^{\alpha-2}u|v|^{\beta},\\
&-\Delta v+Bv-\frac{1}{2}\triangle(v^{2})v=\frac{2\beta}{\alpha+\beta}|u|^{\alpha}|v|^{\beta-2}v,\endaligned\right.
\end{equation}
where $u := u(x) \to 0$ and $v := v(x)\to 0$ as $|x| \to \infty$;  $\alpha,\beta>1$, $2<\alpha+\beta< \frac{4N}{N-2}$, $B>0$ is a constant. In recent years, much attention has been devoted to the quasilinear Schr\"{o}dinger equation of the form
\begin{equation}\label{eq12}
-\Delta u+V(x)u-ku\Delta(u^{2})=|u|^{p-2}u,\ x\in \mathbb R^{N}.
\end{equation}
The solutions of (\ref{eq12}) are related to the existence of standing waves of the following quasilinear equation
\begin{equation}\label{eq13}
i\partial_{t}z=-\Delta z+V(x)z-l(|z|^{2})z-k\Delta g(|z|^{2})g'(|z|^{2})z,\ x\in \mathbb R^{N},
\end{equation}
where $V$ is a given potential, $k\in \mathbb R$, $l$ and $g$ are real functions. For more detailed mathematical and physical background of the problem (\ref{eq12}),  we refer to \cite{bg,bhs,bmmlb,1bl,bl,cs,0s} and the references therein.
\vskip4pt
There has been increasing attention to problem  (\ref{eq12}). In \cite{lw,lw2}, with the help of a constrained minimization argument, the existence of positive ground state solution   was proved. Furthermore, Liu and Wang \cite{lw2}  proved that equation (\ref{eq12}) has a ground state solution by using a change of variables and treating the new problem in an Orlicz space when $4\leq p< \frac{4N}{N-2}$ and the potential $V(x)\in C(\mathbb R^{N}, \mathbb R)$ satisfies
$$(V)\inf\limits_{x\in\mathbb R^{N}}V(x)\geq a>0,\ \forall\ M>0,\hbox{meas}\{x\in \mathbb R^{N}\ | \ V(x)\leq M\}<+\infty.$$
Such kind of hypotheses was introduced to guarantee the compactness of embedding  of
$E:=\left\{u\in H^{1}(\mathbb R^{N})\  | \ \int_{\mathbb R^{N}}V(x)u^{2}<\infty\right\}\hookrightarrow L^{s}(\mathbb R^{N})$
for $ 2\leq s< \frac{2N}{N-2}$.   In \cite{lww}, Liu et al. established the existence of  ground state solutions for (\ref{eq12}) by the Nehari method.  Very recently, under a Nehari-Poho\v{z}aev type constraint and concentration-compactness lemma, Ruiz and Siciliano \cite{rs} showed equation (\ref{eq12}) has a ground state solution for $N\geq3,$ $2<p<\frac{4N}{N-2}$.  Moreover, the potential $V (x)$ satisfies the following conditions

$(V_{1})\ V\in C^{1}(\mathbb R^{N},\mathbb R^{+}), 0<V_{0}\leq V(x)\leq V_{\infty}=\lim\limits_{|x|\rightarrow\infty}V(x)<+\infty;$

$(V_{2})\ \nabla V(x)\cdot x \in L^{\infty}(\mathbb R^{N});\ \ \ \ \ \  \ \ \ \ \ \ \ \ \  \ \ \ \ \ \ \ \ \ \  \ \ \ \ \ \ \ \ $

$(V_{3})$ the map $t\mapsto t^{\frac{N+2}{N+p}}V(t^{\frac{1}{N+p}}x)$ is concave for any $x\in \mathbb R^{N}.\ \ \  \ \ \ \  $

\noindent Wu and Wu \cite{ww} obtained the existence of radial solutions for (\ref{eq12}) by using change of variables when $2<p< \frac{4N}{N-2}$ and the potential $V (x)$ satisfies the similar assumptions as those in Ruiz and Siciliano \cite{rs}. For Schr\"{o}dinger system with quasilinear term, few results is known.
Guo and Tang \cite{gt} studied the following system
$$\left\{\begin{array}{ll}
\aligned &-\Delta u+(\lambda a(x)+1)u-\frac{1}{2}\triangle(u^{2})u=\frac{2\alpha}{\alpha+\beta}|u|^{\alpha-2}u|v|^{\beta},\\
&-\Delta v+(\mu b(x)+1)v-\frac{1}{2}\triangle(v^{2})v=\frac{2\beta}{\alpha+\beta}|u|^{\alpha}|v|^{\beta-2}v,\endaligned
\end{array}\right.$$
and proved the existence of positive ground state solution for $4\leq\alpha+\beta< \frac{4N}{N-2}$.  Moreover, the potential $a(x),b(x)$ satisfy

$(a_{3})\ 0\leq a(x), b(x)\in C(\mathbb R^{N},\mathbb R)$, $\Omega:=\hbox{int}\{a^{-1}(0)\}=\hbox{int}\{b^{-1}(0)\}$ is nonempty with smooth boundary and $\bar{ \Omega}= a^{-1}(0)=b^{-1}(0);$

$(a_{4})\ \hbox{there exist}\ M_{1},M_{2}>0$ such that $meas(\{x\in\mathbb R^{N}\ | \ a(x)\leq M_{1}\})<\infty$ and $meas(\{x\in \mathbb R^{N}\ | \ b(x)\leq M_{2}\})<\infty.$

\noindent Later, a class of gradient system with the form
$$\left\{\begin{array}{ll}
\aligned &-\Delta u+ V_1(x) u- \triangle(u^{2})u= \frac{\partial}{\partial u} F(x,u,v),\\
&-\Delta v+ V_2(x) v- \triangle(v^{2})v=  \frac{\partial}{\partial v} F(x,u,v)\endaligned
\end{array}\right.$$
has been studied by  Severe and Silva \cite{ss}.  In which, besides some other conditions on $V_1(x)$,  $V_2(x)$ and $F(x,u,v)$, the authors obtain the existence of nontrivial solution for the case of $F(x,u,v) = \frac2{\alpha +\beta} |u|^\alpha|v|^\beta$ with $\alpha >2$, $\beta >2$ and $\alpha +\beta < \frac{4N}{N-2}$.   More works on single quasilinear equation or quasilinear system can be found in \cite{acm,jw, m,s,su,ss}.  Note that, in \cite{gt,lw}, by using change of variables, equation is transformed into a related semilinear problem, then one can apply the well-known arguments of \cite{0bl}.

\vskip4pt
To the best of our knowledge, there is no results on the existence of positive ground state solutions for (\ref{eq11}) with $2< \alpha + \beta < \frac{4N}{N-2}$. The main purpose of the present paper is to give a condition to ensure the existence of ground state solution to (\ref{eq11}) for all $\alpha > 1$, $\beta >1$ and $\alpha + \beta < \frac{4N}{N-2}$.
Before state our main results, we make the following assumptions

$(A_{1})$ $A\in C^{1}(\mathbb R^{N},\mathbb R^{+}),$ $0<A_{0}\leq A(x)\leq A_{\infty}=\lim\limits_{|x|\rightarrow\infty}A(x)<+\infty$;

$(A_{2})$ $ \nabla A(x)\cdot x\in L^{\infty}(\mathbb R^{N});$

$(A_{3})$ the map $s\mapsto s^{\frac{N+2}{N+\alpha+\beta}}A(s^{\frac{1}{N+\alpha+\beta}}x)$ is concave for any $x\in \mathbb R^{N}$.

\noindent It is worth noting that the similar hypotheses on $A(x) $ as above $(A_{1})$-$(A_{3})$ once appeared in \cite{rs,ww} to study  the single quasilinear Schr\"{o}dinger equation, where the condition $(A_{1})$ is used to derive the existence of a strongly convergent subsequence. While for the system,  we only need condition like $(A_{1})$ in one equation and the coefficient $B$ being a constant. This seems to be a different phenomenon due to the coupling of $u$ and $v$.
\vskip4pt

 Our main result  reads as follows.

\begin{Theorem}\label{th11}
Under the assumptions  $(A_{1})$-$(A_{3})$. If $ \alpha,\beta>1,\alpha+\beta\in(2, \frac{4N}{N-2})$, then problem (\ref{eq11}) has a positive ground state solution.
\end{Theorem}

\begin{Corollary}\label{co12}
If $A(x)$ is a positive constant,  one can still obtain the same results as  Theorem \ref{th11} for system (\ref{eq11}).
\end{Corollary}

\noindent {\bf Remark 1.3.}  As a main novelty with respect to some results in \cite{gt,lw,lww}, we are able to deal with exponents $\alpha+\beta\in(2, \frac{4N}{N-2}).$
\vskip4pt

 \vskip4pt
The rest of the paper is organized as follows. In Section 2, we state the variational framework of our problem and establish some preliminary results. Theorem \ref{th11} is proved in Section 3.

\mysection {Preliminaries}
Let $H^{1}(\mathbb R^{N})$ be the usual Sobolev space. Define $X :=\ H\times H$
with  $H := \{ u \in H^{1}(\mathbb R^{N})\   | \ \int_{\mathbb R^{N}}u^{2}|\nabla u|^{2}<+\infty  \}.$
The term $\int_{\mathbb R^{N}}u^{2}|\nabla u|^{2}$ is not convex and $H$ is not even a vector space. So, the usual min-max techniques cannot be directly applied, nevertheless $H$ is a complete metric space with distance
$$ d_{H}( u, \omega  )=\|u-\omega\|_{H^{1}}+|\nabla u^{2}-\nabla \omega^{2}|_{L^{2}}.  $$
We define $$\aligned d_{X}((u,v),(\omega,\nu))
&=\|u-\omega\|_{H^{1}}+|\nabla u^{2}-\nabla \omega^{2}|_{L^{2}}+\|v-\nu\|_{H^{1}}\\
&\ \ \ +|\nabla v^{2}-\nabla \nu^{2}|_{L^{2}}. \endaligned$$
Then we call $(u,v)\in X$ is a weak solution of (\ref{eq11}) if for any $\varphi_{1},\varphi_{2}\in C_{0}^{\infty}(\mathbb R^{N})$,
$$\aligned &\int_{\mathbb R^{N}}\bigg((1+u^{2})\nabla u\nabla \varphi_{1}+u|\nabla u|^{2}\varphi_{1}+A(x)u\varphi_{1}\\
&\ \ \ \ -\frac{2\alpha}{\alpha+\beta}|u|^{\alpha-2}u|v|^{\beta}\varphi_{1}\bigg)=0,\endaligned$$
and
$$\aligned &\int_{\mathbb R^{N}}\bigg((1+v^{2})\nabla v\nabla \varphi_{2}+v|\nabla v|^{2}\varphi_{2}+B v\varphi_{2}\\
&\ \ \ \  -\frac{2\beta}{\alpha+\beta}|u|^{\alpha}|v|^{\beta-2}v\varphi_{2}\bigg)=0.\endaligned$$
Hence there is a one-to-one correspondence between solutions of (\ref{eq11}) and critical points of the following  functional $I: X \rightarrow\mathbb R$ defined by
\begin{equation}\label{eq21}
\aligned I(u,v)&=\int_{\mathbb R^{N}}\bigg(\frac{1}{2}(|\nabla u|^{2}+|\nabla v|^{2}+A(x)u^{2}+B v^{2}\\
&\ \ \ \ +u^{2}|\nabla u|^{2}+v^{2}|\nabla v|^{2})-\frac{2}{\alpha+\beta} |u|^{\alpha}|v|^{\beta}\bigg).\endaligned
\end{equation}
For any $\varphi_{1},\varphi_{2}\in C_{0}^{\infty}(\mathbb R^{N})$, $(u,v)\in X,$ and $(u,v)+(\varphi_{1},\varphi_{2})\in X$, we compute the Gateaux derivative
$$\aligned &\ \langle I'(u,v),(\varphi_{1},\varphi_{2})\rangle\\
=&\ \int_{\mathbb R^{N}}\bigg((1+u^{2})\nabla u\nabla \varphi_{1}+(1+v^{2})\nabla v\nabla \varphi_{2}+u|\nabla u|^{2}\varphi_{1}\\
&\ +v|\nabla v|^{2}\varphi_{2}+A(x)u\varphi_{1}+B v\varphi_{2}-\frac{2\alpha}{\alpha+\beta}|u|^{\alpha-2}u|v|^{\beta}\varphi_{1}\\
&\  -\frac{2\beta}{\alpha+\beta}|v|^{\beta-2}v|u|^{\alpha}\varphi_{2}\bigg).\endaligned$$
Then, $(u,v)\in X $ is a solution of (\ref{eq11}) if and only if
$$\langle I'(u,v),(\varphi_{1},\varphi_{2})\rangle=0.$$

For convenience, throughout this paper,  $C, C_{i}(i=1,2,\ldots)$ denotes (possibly different) positive constants, $\int_{\mathbb R^{N}}g$ denotes the integral $\int_{\mathbb R^{N}}g(z)dz$ and $\|(u,v)\|:=\int_{\mathbb R^{N}}(|\nabla u|^{2}+|\nabla v|^{2}+u^{2}+v^{2})$.

\begin{Lemma}\label{le21} (Poho\v{z}aev identity) Under the assumptions $(A_{1})$-$(A_{3})$,  if $(u,v)\in X$ is a weak solution of problem (\ref{eq11}), then $(u,v)$ satisfies $P(u,v)=0,$ where
$$ \aligned P(u,v):&=\frac{N-2}{2}\int_{\mathbb R^{N}}|\nabla u|^{2}+\frac{N-2}{2}\int_{\mathbb R^{N}}|\nabla v|^{2}+\frac{N}{2}\int_{\mathbb R^{N}}A(x)u^{2}\\
&\ \ \   +\frac{N}{2}\int_{\mathbb R^{N}}B v^{2}+\frac{1}{2}\int_{\mathbb R^{N}}\nabla A(x)\cdot xu^{2}+\frac{N-2}{2}\int_{\mathbb R^{N}}u^{2}|\nabla u|^{2}\\
&\ \ \ +\frac{N-2}{2}\int_{\mathbb R^{N}}v^{2}|\nabla v|^{2}-\frac{2N}{\alpha+\beta}\int_{\mathbb R^{N}}|u|^{\alpha}|v|^{\beta}.\endaligned $$
\end{Lemma}

\begin{proof} The proof is standard, and we omit the details.
\end{proof}

For any $u\in H$, we define $u_{t}:\mathbb R^{+}\rightarrow H$ by
$$ u_{t}(x):=tu(t^{-1}x).$$
Let $t\in \mathbb R^{+}$ and $(u,v)\in X$, we have
$$\aligned I(u_{t},v_{t})&=\frac{t^{N}}{2}\int_{\mathbb R^{N}}|\nabla u|^{2}+\frac{t^{N}}{2}\int_{\mathbb R^{N}}|\nabla v|^{2}+\frac{t^{N+2}}{2}\int_{\mathbb R^{N}}A(tx)u^{2}\\
&\ \ \ \ +\frac{t^{N+2}}{2}\int_{\mathbb R^{N}}B v^{2}+\frac{t^{N+2}}{2}\int_{\mathbb R^{N}}u^{2}|\nabla u|^{2}\\
&\ \ \ \ +\frac{t^{N+2}}{2}\int_{\mathbb R^{N}}v^{2}|\nabla v|^{2}-\frac{2t^{N+\alpha+\beta}}{\alpha+\beta}\int_{\mathbb R^{N}}|u|^{\alpha}|v|^{\beta}.\endaligned$$
Denote $h_{uv}(t):=I(u_{t},v_{t})$. Since $\alpha+\beta>2,$ we see that $h_{uv}(t)>0$ for $t>0$ small enough and $h_{uv}(t)\rightarrow -\infty$ as $t\rightarrow +\infty,$ this implies that $h_{uv}(t)$ attains its maximum.

Moreover, by $(A_{2}),$  $h_{uv}(t): \mathbb R^{+}\rightarrow \mathbb R$ is $C^{1}$ and
$$\aligned h'_{uv}(t)
&=\frac{N}{2}t^{N-1}\int_{\mathbb R^{N}}(|\nabla u|^{2}+|\nabla v|^{2})+\frac{N+2}{2}t^{N+1}\int_{\mathbb R^{N}}(A(tx)u^{2}\\
&\ \ \ \ +B v^{2}+u^{2}|\nabla u|^{2}+v^{2}|\nabla v|^{2}) +\frac{ t^{N+1}}{2}\int_{\mathbb R^{N}} \nabla A(tx)\cdot tx u^{2} \\
&\ \ \ \ -\frac{2(N+\alpha+\beta)}{\alpha+\beta}t^{N+\alpha+\beta-1}\int_{\mathbb R^{N}}|u|^{\alpha}|v|^{\beta}.\endaligned $$

Motivated by \cite{rs}, we introduce the following set
$$\mathcal{M}=\{(u,v)\in X\backslash\{(0,0)\}\ | \ G(u,v)=0\},$$
where $G:X\rightarrow \mathbb R$ and
$$ \aligned G(u,v)&=\frac{N}{2} \int_{\mathbb R^{N}}(|\nabla u|^{2}+|\nabla v|^{2})+\frac{N+2}{2} \int_{\mathbb R^{N}}(A(x)u^{2}+B v^{2})\\
&\ \ \ \ +\frac{N+2}{2} \int_{\mathbb R^{N}}(u^{2}|\nabla u|^{2}+v^{2}|\nabla v|^{2})+\frac{1}{2}\int_{\mathbb R^{N}} \nabla A(x)\cdot x u^{2}\\
&\ \ \ \ -\frac{2(N+\alpha+\beta)}{\alpha+\beta}\int_{\mathbb R^{N}}|u|^{\alpha}|v|^{\beta}. \endaligned $$
Observe that if $(u,v)\in X$ is a weak solution of problem (\ref{eq11}), then $\langle I'(u,v),(u,v)\rangle=0$ and $P(u,v)=0$, therefore, $G(u,v)=0.$

Define
\begin{equation}\label{eq22}
m=\inf\limits_{(u,v)\in \mathcal{M}}I(u,v).
\end{equation}
Then our aim is to prove that $m$ is achieved, i.e., there exists $(u,v)\in \mathcal{M}$, such that $I(u,v)=m.$ In the rest of this section, we will give some properties of the set $\mathcal{M}$ and show that (\ref{eq22}) is well defined.
\begin{Lemma}\label{le22} For any $(u,v)\in X,u\neq0$ and $v\neq0,$ there is a unique $\bar{ t}:=t(u,v)>0$ such that $h_{uv}$ attains its maximum at $\bar{t}$. Moreover, $h_{uv}$ is positive and increasing for $t\in [0,\bar{t}]$ and decreasing for $t>\bar{t}$. Moreover, for any $u,v\neq0,$
$$m=\inf_{(u,v)\in X}\max_{t>0}I(u_{t},v_{t}). $$
\end{Lemma}
\begin{proof}
For any $t>0$, set $s=t^{N+\alpha+\beta}$, we obtain
$$\aligned h_{uv}(s)&=\frac{s^{\frac{N}{N+\alpha+\beta}}}{2}\int_{\mathbb R^{N}}|\nabla u|^{2}+\frac{s^{\frac{N}{N+\alpha+\beta}}}{2}\int_{\mathbb R^{N}}|\nabla v|^{2}+\frac{s^{\frac{N+2}{N+\alpha+\beta}}}{2}\int_{\mathbb R^{N}}u^{2}|\nabla u|^{2}\\
&\ \ \ \ +\frac{s^{\frac{N+2}{N+\alpha+\beta}}}{2}\int_{\mathbb R^{N}}v^{2}|\nabla v|^{2}+\frac{s^{\frac{N+2}{N+\alpha+\beta}}}{2}\int_{\mathbb R^{N}}A(s^{\frac{1}{N+\alpha+\beta}}x)u^{2}\\
&\ \ \ \ +\frac{s^{\frac{N+2}{N+\alpha+\beta}}}{2}\int_{\mathbb R^{N}}Bv^{2}-\frac{2s }{\alpha+\beta}\int_{\mathbb R^{N}}|u|^{\alpha}|v|^{\beta}.\endaligned $$
This is a concave function under the condition $(A_{3})$. We already know that it attains its maximum, let $\bar{t}$ be the unique point at which this maximum is achieved. Hence, $\bar{t}$ is the unique critical point of $h_{uv}$ and $h_{uv}$ is positive and increasing for $0<t <\bar{t}$ and decreasing for $t>\bar{t}$. In particular, for any $u,v\neq0$, $\bar{t}\in \mathbb R$ is the unique value such that $(u_{\bar{t}},u_{\bar{t}})\in\mathcal{M},$ and $I(u_{\bar{t}},v_{\bar{t}})$ reaches a global maximum for $t=\bar{t}$. This finishes the proof.
\end{proof}
It follows from Lemma \ref{le22}  that $\mathcal{M}\neq\emptyset $ and (\ref{eq22}) is well defined.
\vskip4pt
\begin{Lemma}\label{le23}
 $m>0$.
\end{Lemma}

\begin{proof}  Let us define
$$\aligned \bar{I}(u,v)&=\int_{\mathbb R^{N}}\bigg(\frac{1}{2}(|\nabla u|^{2}+|\nabla v|^{2}+ A_{0}u^{2}+Bv^{2}\\
&\ \ \ \ +u^{2}|\nabla u|^{2}+v^{2}|\nabla v|^{2}) -\frac{2}{\alpha+\beta} |u|^{\alpha}|v|^{\beta}\bigg),\endaligned$$
where $A_{0}$ comes from $(A_{1})$. Obviously, $\bar{I}(u,v)\leq I(u,v)$, and this implies that for any $u,v\neq0$,
$$\bar{m}:=\inf_{(u,v)\in X}\max_{t>0}\bar{I}(u_{t},v_{t})\leq \inf_{(u,v)\in X}\max_{t>0}I(u_{t},v_{t})=m.$$
Then, it suffices to show that $\bar{m}>0.$ Define
$$ \bar{\mathcal{M}}=\{(u,v)\in X\backslash\{(0,0)\} \ | \ g'_{uv}(1)=0\},$$
where $g_{uv}(t)=\bar{I}(u_{t},v_{t}).$ By Lemma \ref{le22} applied to $A\equiv A_{0}$, we get that
$$\bar{m }=\inf_{(u,v)\in\bar{\mathcal{M}}}\bar{I}(u,v).$$
For every $(u,v)\in\bar{\mathcal{ M}}$, by using H\"{o}lder, Young and interpolation inequalities,
$$\aligned  &\ \frac{N+2}{2}\int_{\mathbb R^{N}}(A_{0}u^{2}+Bv^{2})+\frac{N+2}{2}\int_{\mathbb R^{N}}(u^{2}|\nabla u|^{2}+v^{2}|\nabla v|^{2})\\
\leq&\ \frac{2(N+\alpha+\beta)}{\alpha+\beta}\int_{\mathbb R^{N}}|u|^{\alpha}|v|^{\beta}\\
\leq&\ \frac{2(N+\alpha+\beta)}{\alpha+\beta}\bigg(\int_{\mathbb R^{N}}|u |^{\alpha+\beta}\bigg)^{\frac{\alpha}{\alpha+\beta}}\bigg(\int_{\mathbb R^{N}}|v |^{\alpha+\beta}\bigg)^{\frac{\beta}{\alpha+\beta}}\\
\leq&\ \frac{2\alpha(N+\alpha+\beta)}{(\alpha+\beta)^{2}} \int_{\mathbb R^{N}}|u |^{\alpha+\beta} +\frac{2\beta(N+\alpha+\beta)}{(\alpha+\beta)^{2}} \int_{\mathbb R^{N}}|v |^{\alpha+\beta} \\
\leq&\ \frac{2\alpha(N+\alpha+\beta)}{(\alpha+\beta)^{2}}\bigg(\int_{\mathbb R^{N}}u^{2}\bigg)^{\frac{l(\alpha+\beta)}{2}}\bigg(\int_{\mathbb R^{N}}(|u|^{\frac{4N}{N-2}}\bigg)^{\frac{(1-l)(N-2)(\alpha+\beta)}{4N}}\\
&\  +\frac{2\beta(N+\alpha+\beta)}{(\alpha+\beta)^{2}}\bigg(\int_{\mathbb R^{N}}v^{2}\bigg)^{\frac{l(\alpha+\beta)}{2}}\bigg(\int_{\mathbb R^{N}}(|v|^{\frac{4N}{N-2}}\bigg)^{\frac{(1-l)(N-2)(\alpha+\beta)}{4N}}\\
\leq&\ \frac{N+2}{2}\int_{\mathbb R^{N}}(A_{0}u^{2}+Bv^{2})+C\int_{\mathbb R^{N}}(|u|^{\frac{4N}{N-2}}+|v|^{\frac{4N}{N-2}}).\endaligned$$
 So, by the Sobolev inequality,
$$\aligned \frac{N+2}{2}\int_{\mathbb R^{N}}(u^{2}|\nabla u|^{2}+v^{2}|\nabla v|^{2})&\leq C\int_{\mathbb R^{N}}(|u|^{\frac{4N}{N-2}}+|v|^{\frac{4N}{N-2}})\\
&\leq C_{1}\bigg(\int_{\mathbb R^{N}}u^{2}|\nabla u|^{2}+v^{2}|\nabla v|^{2}\bigg)^{\frac{N}{N-2}},\endaligned $$
which implies that $\int_{\mathbb R^{N}}(u^{2}|\nabla u|^{2}+v^{2}|\nabla v|^{2})$ is bounded away from zero on $\bar{\mathcal{ M}}$, then
$$\aligned\bar{I}(u,v)&= \frac{\alpha+\beta}{2(N+\alpha+\beta)}\int_{\mathbb R^{N}}|\nabla u|^{2}+\frac{\alpha+\beta}{2(N+\alpha+\beta)}\int_{\mathbb R^{N}}|\nabla v|^{2}\\
&\ \ \ \ +\frac{A_{0}(\alpha+\beta-2)}{2(N+\alpha+\beta)}\int_{\mathbb R^{N}}u^{2}+\frac{B(\alpha+\beta-2)}{2(N+\alpha+\beta)}\int_{\mathbb R^{N}}v^{2}\\
&\ \ \ \ +\frac{\alpha+\beta-2}{N+\alpha+\beta}\int_{\mathbb R^{N}}u^{2}|\nabla u|^{2}+\frac{\alpha+\beta-2}{N+\alpha+\beta}\int_{\mathbb R^{N}}v^{2}|\nabla v|^{2}\\
&>0.\endaligned $$
\end{proof}

\begin{Lemma}\label{le24}
For any $(u,v)\in  \mathcal{M},$ there exists a constant $c>0$ such that
$$  I(u,v)\geq c\int_{\mathbb R^{N}}(u^{2}+ v^{2}+|\nabla u|^{2}+|\nabla v|^{2}+ u^{2}|\nabla u|^{2}+v^{2}|\nabla v|^{2}).$$
\end{Lemma}
\begin{proof} Take $(u,v)\in \mathcal{M}$ and choose $t\in(0, 1).$ We compute

$$\aligned &\ I(u_{t},v_{t})-t^{N+\alpha+\beta}I(u,v)\\
=&\ \int_{\mathbb R^{N}}\bigg(\frac{t^{N}}{2}-\frac{t^{N+\alpha+\beta}}{2}\bigg)(|\nabla v|^{2}+|\nabla v|^{2})\\
&\ +\int_{\mathbb R^{N}}\bigg(\frac{t^{N+2}}{2}-\frac{t^{N+\alpha+\beta}}{2}\bigg)(u^{2}|\nabla u|^{2}+v^{2}|\nabla v|^{2})\\
&\ +\int_{\mathbb R^{N}}\bigg(\frac{t^{N+2}}{2}A(tx)-\frac{t^{N+\alpha+\beta}}{2}A(x)\bigg)u^{2}\\
&\ +\int_{\mathbb R^{N}}\bigg(\frac{t^{N+2}}{2}-\frac{t^{N+\alpha+\beta}}{2}\bigg)Bv^{2}.\endaligned$$
By $(A_{1})$, there exists $\delta_{1}\in(0, 1)$ depending only on $A_{0}$ and $A_{\infty}$ such that $A(tx)\geq A_{0}\geq\delta_{1} A_{\infty}\geq \delta_{1} A(x)$,  $B\geq\delta_{2}B$ for some positive constant $\delta_{2}\in(0, 1)$ depending only on $B$. Choosing a smaller $t$, we obtain that
$$\frac{t^{N+2}}{2}A(tx)-\frac{t^{N+\alpha+\beta}}{2}A(x)\geq \bigg(\delta_{1}\frac{t^{N+2}}{2}-\frac{t^{N+\alpha+\beta}}{2}\bigg)A(x)\geq\gamma_{1},  $$
$$\frac{t^{N+2}}{2}B-\frac{t^{N+\alpha+\beta}}{2}B\geq B\bigg(\delta_{2}\frac{t^{N+2}}{2}-\frac{t^{N+\alpha+\beta}}{2}\bigg)\geq\gamma_{2}
\ \ \ \ \ \ \ \ \ \ \ \ $$
for a fixed constant $\gamma=\min\{\gamma_{1},\gamma_{2}\}>0$. By Lemma \ref{le22}, $I(u_{t},v_{t})\leq I(u,v)$ and then
$$\aligned &\ (1-t^{N+\alpha+\beta})I(u,v)\\
\geq&\ I(u_{t},v_{t})-t^{N+\alpha+\beta}I(u,v)\\
\geq& \ \gamma\int_{\mathbb R^{N}} (u^{2}+ v^{2}+|\nabla u|^{2}+|\nabla v|^{2}+u^{2}|\nabla u|^{2}+v^{2}|\nabla v|^{2} ).\endaligned$$
We conclude by taking a smaller $\gamma$ and choosing $c=\frac{\gamma }{1-t^{N+\alpha+\beta}}$.
\end{proof}

\mysection{Proof of Theorem \ref{th11} }

In this section, we prove the existence of positive ground state solutions of Nehari-Poho\v{z}aev type to (\ref{eq11}).

\begin{Lemma}\label{le31}
Let $ R>0, \ q\in[2,\frac{4N}{N-2}).$ If $ \{u_{n}\}$ is bounded in $H$ and
$$ \lim_{n\rightarrow\infty}\sup_{y\in \mathbb R^{N}}\int_{B_{R}(y)}|u_{n}|^{q}=0,$$
then we have $u_{n}\rightarrow0$ in $ L^{p}(\mathbb R^{N})$ for $p\in(2,\frac{4N}{N-2}).$
\end{Lemma}
\begin{proof}  The proof is similar to the proof of  \cite[Lemma 2.2]{wz}.
\end{proof}
\begin{Lemma}\label{le332}
 Let  $u_{n}\rightharpoonup u, v_{n}\rightharpoonup v$  in $H^{1}(\mathbb R^{N})$, $u_{n}\rightarrow u, v_{n}\rightarrow v$ a.e in $\mathbb R^{N}$. Then
$$
\lim_{n\rightarrow\infty}\int_{\mathbb R^{N}}|u_{n}|^{\alpha}|v_{n}|^{\beta}-\int_{\mathbb R^{N}}|u|^{\alpha}|v|^{\beta}
=\lim_{n\rightarrow\infty}\int_{\mathbb R^{N}}|u_{n}-u|^{\alpha}|v_{n}-v|^{\beta}.
$$

\end{Lemma}
\begin{proof}
For $n=1,\ 2,\ldots$, we have that
$$\begin{aligned}
&\ \int_{\mathbb R^{N}}|u_n|^\alpha |v_n|^\beta -\int_{\mathbb R^{N}}|u_n-u|^\alpha |v_n-v|^\beta\\
=&\int_{\mathbb R^{N}}(|u_n|^\alpha-|u_n-u|^\alpha) |v_n|^\beta+\int_{\mathbb R^{N}}|u_n-u|^\alpha (|v_n|^\beta-|v_n-v|^\beta).
\end{aligned}$$
Since $u_n\rightharpoonup u, v_n\rightharpoonup v$ in $H^1 (\mathbb{R}^N)$, from \cite[Lemma 2.5]{mj}, one has
$$\int_{\mathbb R^{N}}(|u_n|^\alpha-|u_n-u|^\alpha-|u|^\alpha)^{\frac{p}{\alpha}}\rightarrow0,\ \ n\rightarrow\infty,$$
which means that $|u_n|^\alpha-|u_n-u|^\alpha \rightarrow |u|^\alpha $  in $L^{\frac{p}{\alpha}} (\mathbb{R}^N)$.
Using $|v_n|^\beta \rightharpoonup |v|^\beta$ in $L^{\frac{p}{\beta}} (\mathbb{R}^N)$, it follows from $\alpha +\beta = p$ that
$$\int_{\mathbb R^{N}}(|u_n|^\alpha-|u_n-u|^\alpha) |v_n|^\beta\rightarrow\int_{\mathbb R^{N}}|u|^\alpha |v|^\beta,\ \ n\rightarrow\infty.$$
Similarly, $|v_n|^\beta-|v_n-v|^\beta \rightarrow |v|^\beta $  in $L^{\frac{p}{\beta}} (\mathbb{R}^N)$.
As $|u_n-u|^\alpha \rightharpoonup 0$ in $L^{\frac{p}{\alpha}} (\mathbb{R}^N)$, we obtain that
$$\int_{\mathbb R^{N}}|u_n-u|^\alpha (|v_n|^\beta-|v_n-v|^\beta)\rightarrow0,\ \ n\rightarrow\infty.$$
This proves the lemma.
\end{proof}
\begin{Lemma}\label{le32}
$m$ is achieved at some $(u,v)\in \mathcal{M}$.
\end{Lemma}

\begin{proof} Let $(u_{n},v_{n})\in \mathcal{M}$ so that $I(u_{n},v_{n})\rightarrow m$. By Lemma \ref{le24}, $\{u_{n}\}$, $\{v_{n}\}$, $\{u_{n}^{2}\}$  and $\{v_{n}^{2}\}$ are bounded in $H^{1}(\mathbb R^{N}).$ Then, there exist subsequence of $\{u_{n}\}$, $\{v_{n}\}$ (still denoted by $\{u_{n}\}$, $\{v_{n}\}$) such that $u_{n}\rightharpoonup u$ and $u_{n}^{2}\rightharpoonup u^{2}$ in $H^{1}(\mathbb R^{N})$, $v_{n}\rightharpoonup v$ and $v_{n}^{2}\rightharpoonup v^{2}$ in $H^{1}(\mathbb R^{N})$,
which implies that $\{u_{n}\}$ and $\{v_{n}\}$ are bounded in $L^{\alpha+\beta}(\mathbb R^{N})$, $\alpha+\beta\in(2,\frac{4N}{N-2})$. The proof consists of four steps.
\vskip4pt
{\bf Step 1.} $\int_{\mathbb R^{N}}|u_{n}|^{\alpha}|v_{n}|^{\beta}\not\rightarrow0$.
\vskip4pt
It follows from (\ref{eq21}) and Lemma \ref{le23} that
$$ \aligned I(u_{n},v_{n})&=\int_{\mathbb R^{N}}\bigg(\frac{1}{2}(|\nabla u_{n}|^{2}+|\nabla v_{n}|^{2}+A(x)u_{n}^{2}+Bv_{n}^{2}\\
&\ \ \ \ +u_{n}^{2}|\nabla u_{n}|^{2}+v_{n}^{2}|\nabla v_{n}|^{2})-\frac{2}{\alpha+\beta} |u_{n}|^{\alpha}|v_{n}|^{\beta}\bigg)\\
&\rightarrow m>0,\endaligned$$
then $$\|(u_{n},v_{n})\|+\|(u_{n}^{2},v_{n}^{2})\|\not\rightarrow0.$$
 By Lemma \ref{le22}, for $t > 1$,
$$\aligned m&\leftarrow I(u_{n},v_{n})\\
&\geq I((u_{n})_{t},(v_{n})_{t})\\
&=\frac{t^{N}}{2}\int_{\mathbb R^{N}}(|\nabla u_{n}|^{2}+|\nabla v_{n}|^{2})+\frac{t^{N+2}}{2}\int_{\mathbb R^{N}}(A(tx)u_{n}^{2}+Bv_{n}^{2})\\
&\ \ \ \ +\frac{t^{N+2}}{2}\int_{\mathbb R^{N}}(u_{n}^{2}|\nabla u_{n}|^{2}+v_{n}^{2}|\nabla v_{n}|^{2}) -\frac{2t^{N+\alpha+\beta}}{\alpha+\beta}\int_{\mathbb R^{N}}|u_{n}|^{\alpha}|v_{n}|^{\beta}\\
&\geq \frac{t^{N}}{2}\int_{\mathbb R^{N}}(|\nabla u_{n}|^{2}+|\nabla v_{n}|^{2}+A_{0}u_{n}^{2} +Bv_{n}^{2}+u_{n}^{2}|\nabla u_{n}|^{2}+v_{n}^{2}|\nabla v_{n}|^{2})\\
&\ \ \ \ -\frac{2t^{N+\alpha+\beta}}{\alpha+\beta}\int_{\mathbb R^{N}}|u_{n}|^{\alpha}|v_{n}|^{\beta}\\
&\geq \frac{t^{N}}{2}\delta-\frac{2t^{N+\alpha+\beta}}{\alpha+\beta}\int_{\mathbb R^{N}}|u_{n}|^{\alpha}|v_{n}|^{\beta},\endaligned$$
where $\delta$ is a fixed constant. It suffices to chose $t > 1$ so that $ \frac{t^{N}\delta}{2}>2m$ to get a lower bound for $\int_{\mathbb R^{N}}|u_{n}|^{\alpha}|v_{n}|^{\beta}$.

Therefore, we may assume that
\begin{equation}\label{eq31}
\int_{\mathbb R^{N}}|u_{n}|^{\alpha}|v_{n}|^{\beta}\rightarrow D\in(0,\infty).
\end{equation}
\vskip4pt
{\bf Step 2.} Splitting by concentration-compactness.
\vskip4pt
By using (\ref{eq31}) and Lemma \ref{le31}, there exist $\delta>0$ and $\{x_{n}\}\subset \mathbb R^{N}$  such that
\begin{equation}\label{eq031}
\int_{B_{x_{n}}(1)}|u_{n}|^{\alpha+\beta}>\delta>0.
\end{equation}
For $\varepsilon>0$ and set $R>\max\{1, \varepsilon^{-1}\}$, $\eta_{R}(t)$ a smooth function defined on $[0,+\infty)$ such that
\vskip4pt
a) $\eta_{R}(t)=1$ for $0\leq t\leq R,$

b) $\eta_{R}(t)=0$ for $t\geq2R$,

c) $\eta'_{R}(t)\leq\frac{2}{R}.$
\vskip4pt
Define $ \theta_{n}(x)=\eta_{R}(|x-x_{n}|)u_{n}(x)$, $\lambda_{n}(x)=(1-\eta_{R}(|x-x_{n}|))u_{n}(x)$, $\xi_{n}(x)=\eta_{R}(|x-x_{n}|)v_{n}(x)$ and $\mu_{n}(x)=(1-\eta_{R}(|x-x_{n}|))v_{n}(x)$.
Obviously, $(\theta_{n},\xi_{n}),(\lambda_{n},\mu_{n})\in X$ and $(u_{n},v_{n})=( \theta_{n},\xi_{n})+(\lambda_{n},\mu_{n})=( \theta_{n}+\lambda_{n},\xi_{n}+\mu_{n})$. Observe that in particular
\begin{equation}\label{eq32}
\liminf_{n\rightarrow+\infty}\int_{B_{x_{n}}(R)}| \theta_{n}|^{\alpha+\beta}\geq\delta.
\end{equation}

{\bf Step 3.}
There exists $\tau> 0$ independent of $\varepsilon$ and $n_{0}= n_{0}(\varepsilon)$ such that $\|(\lambda_{n},\mu_{n})\|+\|(\lambda_{n}^{2},\mu_{n}^{2})\|\leq C\varepsilon^{\tau}$ for all $ n\geq n_{0}$.
\vskip4pt
Define $(\omega_{n},\nu_{n})=(u_{n}(\cdot+ x_{n}),v_{n}(\cdot+ x_{n}))$, where $\{x_{n}\}$ is given in (\ref{eq031}). Clearly, $\omega_{n}\rightharpoonup \omega $ and $\nu_{n}\rightharpoonup \nu$ in $H^{1}(\mathbb R^{N})$. By taking a larger $R$, we may assume that $\int_{B_{2R\backslash R}}|\omega|^{\alpha}| \nu|^{\beta}<\varepsilon $, where $ B_{2R\backslash R}$ denotes the anulus centered in 0 with radii $R$ and $2R$. Then,
\begin{equation}\label{eq33}
\bigg| \int_{\mathbb R^{N}}|u_{n}|^{\alpha}|v_{n}|^{\beta}-\int_{\mathbb R^{N}}| \theta_{n}|^{\alpha}|\xi_{n}|^{\beta}-\int_{\mathbb R^{N}}|\lambda_{n}|^{\alpha}|\mu_{n}|^{\beta}\bigg|\leq 3\varepsilon.
\end{equation}
Since $|\nabla \omega_{n}|^{2}$, $|\nabla \nu_{n}|^{2}$ are uniformly bounded in  $L^{1}(\mathbb R^{N})$, up to a subsequence,  $|\nabla \omega_{n}|^{2}$,  $|\nabla \nu_{n}|^{2}$ converge (in the sense of measure) to a certain positive measure  $meas(\mathbb R^{N})<+\infty$.  By enlarging $R$, we may assume that $meas(B_{2R\backslash R})<\varepsilon$. Then, for $n$ large enough,
$$\int_{\mathbb R^{N}}|\nabla u_{n}|^{2}\eta_{R}(|x-x_{n}|)(1-\eta_{R}(|x-x_{n}|))<\varepsilon, $$
$$\int_{\mathbb R^{N}}|\nabla v_{n}|^{2}\eta_{R}(|x-x_{n}|)(1-\eta_{R}(|x-x_{n}|))<\varepsilon,$$
\begin{equation}\label{eq34}
\aligned \bigg| \int_{\mathbb R^{N}}(|\nabla u_{n}|^{2} -|\nabla  \theta_{n}|^{2}-|\nabla \lambda_{n}|^{2}) \bigg|=\bigg|2\int_{\mathbb R^{N}}\nabla \theta_{n}\nabla \lambda_{n}\bigg|\leq C\varepsilon,\endaligned
\end{equation}
\begin{equation}\label{eq35}
\aligned \bigg| \int_{\mathbb R^{N}}(|\nabla v_{n}|^{2} -|\nabla \xi_{n}|^{2}-|\nabla \mu_{n}|^{2}) \bigg|=\bigg|2\int_{\mathbb R^{N}}\nabla \xi_{n}\nabla \mu_{n}\bigg|\leq C\varepsilon,\endaligned
\end{equation}
\begin{equation}\label{eq36}
\bigg|\int_{\mathbb R^{N}}(u_{n}^{2}|\nabla u_{n}|^{2} -\theta_{n}^{2}|\nabla  \theta_{n}|^{2}-\lambda_{n}^{2}|\nabla \lambda_{n}|^{2}) \bigg|< C\varepsilon,
\end{equation}
\begin{equation}\label{eq36}
\bigg|\int_{\mathbb R^{N}}(v_{n}^{2}|\nabla v_{n}|^{2} -\xi_{n}^{2}|\nabla \xi_{n}|^{2}-\mu_{n}^{2}|\nabla \mu_{n}|^{2}) \bigg|< C\varepsilon,
\end{equation}
\begin{equation}\label{eq37}
\bigg|\int_{\mathbb R^{N}}(A(tx)u_{n}^{2} -A(tx)\theta_{n}^{2}-A(tx)\lambda_{n}^{2}) \bigg|<C\varepsilon,\ \
\end{equation}
\begin{equation}\label{eq38}
 \bigg|\int_{\mathbb R^{N}}B(v_{n}^{2} -\xi_{n}^{2}-\mu_{n}^{2}) \bigg|< C\varepsilon.\ \ \ \ \ \ \ \ \ \
\end{equation}
By (\ref{eq33})-(\ref{eq38}) we obtain that for $n$ sufficiently large and $t>0$,
\begin{equation}\label{eq39}
\aligned &\ \bigg| I((u_{n})_{t},(v_{n})_{t})-I((\theta_{n})_{t},(\xi_{n})_{t})-I((\lambda_{n})_{t},(\mu_{n})_{t})\bigg|\\
\leq&\ C\varepsilon(t^{N}+t^{N+\alpha+\beta}).\endaligned
\end{equation}
Let us denote with $t_{1},t_{2}>0$  which maximize $h_{\theta_{n}\xi_{n}}(t)$ and $h_{\lambda_{n}\mu_{n}}(t)$ respectively, i.e.,
$$I((\theta_{n})_{t_{1}},(\xi_{n})_{t_{1}})=\max_{t>0}I((\theta_{n})_{t},(\xi_{n})_{t}),\ $$ $$I((\lambda_{n})_{t_{2}},(\mu_{n})_{t_{2}})=\max_{t>0}I((\lambda_{n})_{t},(\mu_{n})_{t}).$$
First, let us assume that $t_{1}\leq t_{2}$. Then,
\begin{equation}\label{eq310}
I((\lambda_{n})_{t},(\mu_{n})_{t})\geq0,\ \   t\leq t_{1}.
\end{equation}
\vskip4pt
{\bf Claim:} There exist $0<\tilde{t}<1<\hat{t}$ independent of  $\varepsilon$ such that $t_{1}\in(\tilde{t},\hat{t})$.

Indeed, take $\hat{t}=\left((\alpha+\beta)D^{-1}L\right)^{\frac{1}{\alpha+\beta-2}}$, where $D$ comes from (\ref{eq31}) and $L$ is large enough such that  $\hat{t}>1$ and moreover
\begin{equation}\label{eq311}
L\geq \int_{\mathbb R^{N}}(A_{\infty}u_{n}^{2}+ B v_{n}^{2}+|\nabla u_{n}|^{2}+|\nabla v_{n}|^{2}+ u_{n}^{2}|\nabla u_{n}|^{2}+v_{n}^{2}|\nabla v_{n}|^{2}),
\end{equation}
then
$$\aligned I((u_{n})_{\hat{t}},(v_{n})_{\hat{t}})&\leq \frac{\hat{t}^{N+2}}{2}\int_{\mathbb R^{N}}(A(\hat{t}x)u_{n}^{2}+ Bv_{n}^{2}+|\nabla u_{n}|^{2}+|\nabla v_{n}|^{2}\\
&\ \ \ \ + u_{n}^{2}|\nabla u_{n}|^{2}+v_{n}^{2}|\nabla v_{n}|^{2}-\frac{4}{\alpha+\beta}\hat{t}^{\alpha+\beta-2}|u_{n}|^{\alpha}|v_{n}|^{\beta})\\
&\leq -L\frac{\hat{t}^{N+2}}{2}\\
&<0.\endaligned $$
It follows from (\ref{eq39}) that
\begin{equation}\label{eq3111}
\aligned &\ I((u_{n})_{t},(v_{n})_{t})\\
\geq&\  I((\theta_{n})_{t},(\xi_{n})_{t})+I((\lambda_{n})_{t},(v_{n})_{t})-C\varepsilon,\ \ t\in(0,\hat{t}].\endaligned
\end{equation}
Then, taking a smaller $\varepsilon$,
$$I((\theta_{n})_{\hat{t}},(\xi_{n})_{\hat{t}})+I((\lambda_{n})_{\hat{t}},(v_{n})_{\hat{t}})<0.$$
Then $I((\theta_{n})_{\hat{t}},(\xi_{n})_{\hat{t}})<0$ or  $I((\lambda_{n})_{\hat{t}},(v_{n})_{\hat{t}})<0$. In any case Lemma \ref{le22} implies that $t_{1}<\hat{t}$  (recall that we are assuming $t_{1}\leq t_{2}$).

For the lower bound, take  $\tilde{t}=\left(\frac{m}{L}\right)^{\frac{1}{N}}$, where $L$ is chosen as in (\ref{eq311}). Let us point out that  $\tilde{t}<1$. For any $t\leq \tilde{t}$,
$$\aligned I((u_{n})_{t},(v_{n})_{t})&\leq \frac{\tilde{t}^{N}}{2}\int_{\mathbb R^{N}}(A(\hat{t}x)u_{n}^{2}+ B v_{n}^{2}+|\nabla u_{n}|^{2}+|\nabla v_{n}|^{2}\\
&\ \ \ \ + u_{n}^{2}|\nabla u_{n}|^{2}+v_{n}^{2}|\nabla v_{n}|^{2})\\
&\leq \frac{m}{2L}\int_{\mathbb R^{N}}(A_{\infty}u_{n}^{2}+ B v_{n}^{2}+|\nabla u_{n}|^{2}+|\nabla v_{n}|^{2}\\
&\ \ \ \ + u_{n}^{2}|\nabla u_{n}|^{2}+v_{n}^{2}|\nabla v_{n}|^{2})\\
&\leq\frac{m}{2}.\endaligned $$
By (\ref{eq310}) and (\ref{eq3111}),
\begin{equation}\label{eq313}
\aligned I((u_{n})_{t_{1}},(v_{n})_{t_{1}})&\geq I((\theta_{n})_{t_{1}},(\xi_{n})_{t_{1}})+I((\lambda_{n})_{t_{1}},(\mu_{n})_{t_{1}})-C\varepsilon\\
&\geq m-C\varepsilon,\endaligned
\end{equation}
by choosing a small $\varepsilon$, we have that $I((u_{n})_{t_{1}},(v_{n})_{t_{1}})\geq \frac{m}{2}$. Therefore, $t_{1}>\tilde{t}$. Since $(u_{n},v_{n})\in \mathcal{M}$, $h_{uv}$ reaches its maximum at  $t=1$. Then,
$$m\leftarrow I(u_{n},v_{n})\geq I((u_{n})_{t_{1}},(v_{n})_{t_{1}}), $$
and using (\ref{eq313}) we deduce, for $n$ large,
$$\aligned I((\lambda_{n})_{t},(\mu_{n})_{t})&\leq I((u_{n})_{t_{1}},(v_{n})_{t_{1}})-I((\theta_{n})_{t_{1}},(\xi_{n})_{t_{1}})+C\varepsilon\\
&\leq I(u_{n},v_{n})-m+C\varepsilon\\
&\leq2C \varepsilon,\endaligned$$  for $t\in(0,t_{1})$. Moreover, for any $t\in(0,\tilde{t}),$
$$\aligned 2C\varepsilon&\geq I((\lambda_{n})_{t},(\mu_{n})_{t})\\
&\geq \frac{t^{N+2}}{2}\int_{\mathbb R^{N}}(A(tx)\lambda_{n}^{2}+ B\mu_{n}^{2}+|\nabla \lambda_{n}|^{2}+|\nabla \mu_{n}|^{2}\\
&\ \ \ \ +\lambda_{n}^{2}|\nabla\lambda_{n}|^{2}+\mu_{n}^{2}|\nabla \mu_{n}|^{2})-\frac{2}{\alpha+\beta}t^{N+\alpha+\beta}\int_{\mathbb R^{N}}|\lambda_{n}|^{\alpha}|\mu_{n}|^{\beta}\\
&\geq \frac{t^{N+2}}{2}q_{n}-Kt^{N+\alpha+\beta},\endaligned$$
where
$$q_{n}=\int_{\mathbb R^{N}}(A_{0}\lambda_{n}^{2}+B\mu_{n}^{2}+|\nabla \lambda_{n}|^{2}+|\nabla \mu_{n}|^{2}+ \lambda_{n}^{2}|\nabla\lambda_{n}|^{2}+\mu_{n}^{2}|\nabla \mu_{n}|^{2})$$
is bounded and $K>D$. Observe that $ \frac{t^{N+2}}{2}q_{n}-Kt^{N+\alpha+\beta}=\frac{t^{N+2}}{4}q_{n}$ for $t=\left(\frac{q_{n}}{4K}\right)^{\frac{1}{\alpha+\beta-2}}$. By taking a larger $K$ we may assume that $\left(\frac{q_{n}}{4K}\right)^{\frac{1}{\alpha+\beta-2}}\leq \tilde{t}$. Then, we obtain
$$2C\varepsilon\geq I((\lambda_{n})_{t},(\mu_{n})_{t})\geq \left(\frac{q_{n}}{4K}\right)^{\frac{N+2}{\alpha+\beta-2}}\frac{q_{n}}{4 }\geq Cq_{n}^{\frac{N+\alpha+\beta}{\alpha+\beta-2}},$$
which implies that
\begin{equation}\label{eq314}
\|(\lambda_{n},\mu_{n})\|+\|(\lambda_{n}^{2},\mu_{n}^{2})\|\leq C\varepsilon^{\frac{\alpha+\beta-2}{2(N+\alpha+\beta)}}.
\end{equation}

In the case  $t_{1}>t_{2}$, we can argue similarly to conclude that $\|( \theta_{n},\xi_{n})\|+\|( \theta_{n}^{2},\xi_{n}^{2})\|\leq C\varepsilon^{\frac{\alpha+\beta-2}{2(N+\alpha+\beta)}}$. But, taking small  $\varepsilon$, this contradicts (\ref{eq32}), so (\ref{eq314}) holds.

\vskip4pt
{\bf Step 4:} The infimum of  $I|_{\mathcal{M}}$ is achieved.
\vskip4pt
 By {\bf Step 2},  $(\omega_{n},\nu_{n})=(u_{n}(\cdot+x_{n}),v_{n}(\cdot+x_{n})),$ $\omega_{n}\rightharpoonup \omega$, $\nu_{n}\rightharpoonup \nu$, $\omega_{n}^{2}\rightharpoonup \omega^{2}$ and $\nu_{n}^{2}\rightharpoonup \nu^{2}$ in $H^{1}(\mathbb R^{N})$. Moreover, by compactness, we know that $\omega_{n}\rightarrow \omega$ in $L_{loc}^{2}(\mathbb R^{N}),\nu_{n}\rightarrow  \nu$ in $L_{loc}^{2}(\mathbb R^{N}).$  Then $\omega\neq0$ since, by (\ref{eq32}),
$$ \delta<\liminf_{n\rightarrow+\infty}\int_{\mathbb R^{N}}| \theta_{n}|^{\alpha+\beta}\leq\lim_{n\rightarrow+\infty}\int_{B_{0}(2R)}|\omega_{n}|^{\alpha+\beta}=\int_{B_{0}(2R)}|\omega|^{\alpha+\beta}.$$
Since $(u_{n},v_{n})=( \theta_{n},\xi_{n})+(\lambda_{n},\mu_{n})$, moreover
$$\|(\lambda_{n},\mu_{n})\|+\|(\lambda_{n}^{2},\mu_{n}^{2})\|\leq C\varepsilon^{\frac{\alpha+\beta-2}{2(N+\alpha+\beta)}}.$$
By using H\"{o}lder inequality,
\begin{equation}\label{eq315}
\aligned \int_{\mathbb R^{N}}|u_{n}^{2}- \theta_{n}^{2}|&\leq \int_{\mathbb R^{N}}|\lambda_{n}|(|u_{n}|+| \theta_{n}|)\\
&\leq \bigg(\int_{\mathbb R^{N}}\lambda_{n}^{2}\bigg)^{\frac{1}{2}}\bigg(\int_{\mathbb R^{N}}(|u_{n}|+| \theta_{n}|)^{2}\bigg)^{\frac{1}{2}}\\
&\leq C\varepsilon^{\frac{\alpha+\beta-2}{2(N+\alpha+\beta)}}.
\endaligned
\end{equation}
\begin{equation}\label{eq316}
\aligned \int_{\mathbb R^{N}}|v_{n}^{2}-\xi_{n}^{2}|&\leq \int_{\mathbb R^{N}}|\mu_{n}|(|v_{n}|+|\xi_{n}|)\\
&\leq \bigg(\int_{\mathbb R^{N}}\mu_{n}^{2}\bigg)^{\frac{1}{2}}\bigg(\int_{\mathbb R^{N}}(|v_{n}|+|\xi_{n}|)^{2}\bigg)^{\frac{1}{2}}\\
&\leq C\varepsilon^{\frac{\alpha+\beta-2}{2(N+\alpha+\beta)}}.
\endaligned
\end{equation}
Furthermore,
$$\int_{\mathbb R^{N}}\theta_{n}^{2}\leq \int_{B_{0}(2R)}\omega_{n}^{2}\rightarrow\int_{B_{0}(2R)}\omega^{2}\leq\int_{\mathbb R^{N}}\omega^{2},  $$
$$\int_{\mathbb R^{N}}\xi_{n}^{2}\leq \int_{B_{0}(2R)}\omega_{n}^{2}\rightarrow\int_{B_{0}(2R)}\nu^{2}\leq\int_{\mathbb R^{N}}\nu^{2}.  $$
Combining these estimate with (\ref{eq315}), (\ref{eq316}), we get that
$$\liminf_{n\rightarrow+\infty}\int_{\mathbb R^{N}}\omega_{n}^{2}=\liminf_{n\rightarrow+\infty}\int_{\mathbb R^{N}}u_{n}^{2}\leq \int_{\mathbb R^{N}}\omega^{2}+C\varepsilon^{\frac{\alpha+\beta-2}{2(N+\alpha+\beta)}},$$
$$\liminf_{n\rightarrow+\infty}\int_{\mathbb R^{N}}\nu_{n}^{2}=\liminf_{n\rightarrow+\infty}\int_{\mathbb R^{N}}v_{n}^{2}\leq \int_{\mathbb R^{N}}\nu^{2}+C\varepsilon^{\frac{\alpha+\beta-2}{2(N+\alpha+\beta)}}.\ $$
Since $\varepsilon$ is arbitrary, we obtain that $\omega_{n}\rightarrow \omega,  \nu_{n}\rightarrow  \nu$ in $ L^{2}(\mathbb R^{N})$, by using interpolation inequality, $\omega_{n}\rightarrow \omega\neq0$  in  $L^{\alpha+\beta}(\mathbb R^{N})$, $\alpha+\beta\in[2,\frac{4N}{N-2}) $.
\vskip4pt
We discuss two cases:
\vskip4pt
{\bf Case 1:} $\{x_{n}\} $  is bounded. Assume, passing to a subsequence, that $x_{n}\rightarrow x_{0}$.  In this case $u_{n}\rightharpoonup u,v_{n}\rightharpoonup v$, $u_{n}^{2}\rightharpoonup u^{2}, v_{n}^{2}\rightharpoonup v^{2}$  in $H^{1}(\mathbb R^{N})$, $u_{n}\rightarrow u,v_{n}\rightarrow v$  in $L^{\alpha+\beta}(\mathbb R^{N})$ for  $\alpha+\beta\in[2,\frac{4N}{N-2}) $, where $u=\omega(\cdot-x_{0}),v= \nu(\cdot-x_{0}) $. By using Lemma \ref{le22} and the weak lower semi-continuity of the norm, we obtain that
$$m=\lim_{n\rightarrow+\infty}I( u_{n} ,v_{n} )\geq \liminf_{n\rightarrow+\infty} I((u_{n})_{t},(v_{n})_{t})\geq I(u_{t},v_{t}),\ \ t>0,$$
therefore, $\max\limits_{t>0}I(u_{t},v_{t})=m$, $u_{n}\rightarrow u,v_{n}\rightarrow v$, $u_{n}^{2}\rightarrow u^{2}, v_{n}^{2}\rightarrow v^{2}$ in $H^{1}(\mathbb R^{N})$. In particular, $(u,v)\in \mathcal{M}$ is a minimizer of $I|_{\mathcal{M}}$.
\vskip4pt
{\bf Case 2:} $\{x_{n}\}$ is unbounded. In this case, by $(A_{1})$ and Lebesgue dominated convergence theorem,
$$\aligned \lim_{n\rightarrow+\infty}\int_{\mathbb R^{N}}A(tx)u_{n}^{2}&=\lim_{n\rightarrow+\infty}\int_{\mathbb R^{N}}A(t(x+x_{n}))\omega_{n}^{2}\\
&=A_{\infty}\int_{\mathbb R^{N}}\omega^{2}\\
&\geq\int_{\mathbb R^{N}}A(tx)\omega^{2}\\
&=\lim_{n\rightarrow+\infty}\int_{\mathbb R^{N}}A(tx)\omega_{n}^{2},\endaligned $$
for $t>0$ fixed.  Then, by using Lemma \ref{le22}, Lemma \ref{le332} and the weak lower semi-continuity of the norm, for every $t>0$,
$$\aligned m&=\lim_{n\rightarrow+\infty}I( u_{n} ,v_{n} )\\
&\geq \liminf_{n\rightarrow+\infty} I((u_{n})_{t},(v_{n})_{t})\\
&\geq\liminf_{n\rightarrow+\infty} I((\omega_{n})_{t},(\nu_{n})_{t})\\
&\geq I((\omega)_{t},( \nu)_{t}).\endaligned$$
Taking $t^{z} $ so that $h_{\omega v}(t)=I(\omega_{t},\nu_{t}) $ reaches its maximum, we conclude that $(\omega_{t^{z}},\nu_{t^{z}})\in \mathcal{M}$  and is a minimizer for  $I|_{\mathcal{M}}$.

Based on a general idea used in \cite{lww}, the minimum of  $I|_{\mathcal{M}}$  is indeed a solution of our equation.

\vskip12pt
\noindent{\bf Proof of Theorem \ref{th11}.}
\vskip4pt
 Let $(\tilde{u},\tilde{v})\in \mathcal{M}$  be a minimizer of the functional $I|_{\mathcal{M}}$. By Lemma \ref{le22},
$$I(\tilde{u},\tilde{v})=\inf_{(u,v)\in X }\max\limits_{t>0}I(u_{t},v_{t})=m.$$
We argue by contradiction by assuming that $(\tilde{u},\tilde{v})$ is not a weak solution of (\ref{eq11}). Then, we can chose $\phi_{1},\phi_{2}\in C_{0}^{\infty}(\mathbb R^{N})$  such that
$$\aligned &\ \langle I'(\tilde{u},\tilde{v}),(\phi_{1},\phi_{2})\rangle\\
=& \int_{\mathbb R^{N}}\bigg( \nabla \tilde{u}\nabla \phi_{1}+ \nabla \tilde{v}\nabla \phi_{2}+\nabla(\tilde{u}^{2})\nabla(\tilde{u}\phi_{1})+\nabla(\tilde{v}^{2})\nabla(\tilde{v}\phi_{2}) \\
&\ +A(x)\tilde{u}\phi_{1}+B\tilde{v}\phi_{2}-\frac{2\alpha}{\alpha+\beta}|\tilde{u}|^{\alpha-2}\tilde{u}|\tilde{v}|^{\beta}\phi_{1}\\
&\  -\frac{2\beta}{\alpha+\beta}|\tilde{v}|^{\beta-2}\tilde{v}|\tilde{u}|^{\alpha}\phi_{2}\bigg)\\
<&\ -1.\endaligned$$
Then we fix $\varepsilon> 0$ sufficiently small such that
$$\langle I' (\tilde{u}_{t}+\sigma\phi_{1},\tilde{v}_{t}+\sigma\phi_{2}),(\phi_{1},\phi_{2}) \rangle\leq -\frac{1}{2},\  \ |t-1|,|\sigma|\leq\varepsilon, $$
and introduce a cut-off function $0\leq \zeta\leq1 $ such that $\zeta(t)=1 $ for $|t-1|\leq\frac{\varepsilon}{2}$ and $\zeta(t)=0 $ for $|t-1|\geq\varepsilon. $
For $ t\geq0$, we define
$$\gamma_{1}(t):=\left\{\aligned &\tilde{u}_{t}, &\hbox{if}\ |t-1|\geq\varepsilon,\\
&\tilde{u}_{t}+\varepsilon\zeta(t)\phi_{1}, &\hbox{if}\ |t-1|<\varepsilon,\endaligned\right. $$
$$\gamma_{2}(t):=\left\{\aligned &\tilde{v}_{t}, &\hbox{if}\ |t-1|\geq\varepsilon,\\
&\tilde{v}_{t}+\varepsilon\zeta(t)\phi_{2}, &\hbox{if}\ |t-1|<\varepsilon.\endaligned\right. $$
Note that $ \gamma_{i}(t)$ is a continuous curve in the metric space $ (H,d)$, eventually choosing a smaller $ \varepsilon$, we get that $d_{H}(\gamma_{i}(t),0)>0$  for  $|t-1|<\varepsilon$, $i=1,2$.
\vskip4pt
{\bf Claim: } $\sup\limits_{t\geq0}I(\gamma_{1}(t),\gamma_{2}(t))<m.$

Indeed, if  $|t-1|\geq\varepsilon $, then $I(\gamma_{1}(t),\gamma_{2}(t))=I(\tilde{u}_{t},\tilde{v}_{t})<I(u,v)=m$. If $ |t-1|<\varepsilon$, by using the mean value theorem to the $C^{1}$  map $[0,\varepsilon]\ni \sigma\mapsto I(\tilde{u}_{t}+\sigma\zeta(t)\phi_{1},\tilde{v}_{t}+\sigma\zeta(t)\phi_{2})\in\mathbb R$, we find, for a suitable $\bar{\sigma}\in(0,\varepsilon)$,
$$\aligned &\ I(\tilde{u}_{t}+\sigma\zeta(t)\phi_{1},\tilde{v}_{t}+\sigma\zeta(t)\phi_{2})\\
=&\ I(\tilde{u}_{t},\tilde{v}_{t})+\langle I'(\tilde{u}_{t}+\bar{\sigma}\zeta(t)\phi_{1},\tilde{v}_{t}+\bar{\sigma}\zeta(t)\phi_{2}),(\zeta(t)\phi_{1},\zeta(t)\phi_{2})\rangle\\
\leq&\ I(\tilde{u}_{t},\tilde{v}_{t})-\frac{1}{2}\zeta(t)\\
<&\ m.\endaligned$$
To conclude, we observe that $G(\gamma_{1}(1-\varepsilon),\gamma_{2}(1-\varepsilon))>0$ and $G(\gamma_{1}(1+\varepsilon),\gamma_{2}(1+\varepsilon))<0$.  By the continuity of the map $t\mapsto G(\gamma_{1}(t),\gamma_{2}(t))$  there exists $t_{0}\in(1-\varepsilon,1+\varepsilon) $  such that $G(\gamma_{1}(t_{0}),\gamma_{2}(t_{0}))=0$. Namely,
$$ (\gamma_{1}(t_{0}),\gamma_{2}(t_{0}))=(\tilde{u}_{t_{0}}+\varepsilon\zeta(t_{0})\phi_{1},\tilde{v}_{t_{0}}+\varepsilon\zeta(t_{0})\phi_{2})\in \mathcal{M},$$  and $I(\gamma_{1}(t_{0}),\gamma_{2}(t_{0}))<m$, this is a contradiction.

We already know that the minimizer of $I|_{\mathcal{M}}$  is a solution. Since any solution of (\ref{eq11}) belongs to  $\mathcal{M}$, the minimizer is a ground state. In addition, if $(\tilde{u},\tilde{v})\in \mathcal{M}$ a minimizer for $I|_{\mathcal{M}}$, then, $(|\tilde{u}|,|\tilde{v}|)\in \mathcal{M}$ is also a minimizer, and hence a solution. By the classical maximum principle to each equation of (\ref{eq11}), we get that $\tilde{u},\tilde{v}>0$.
\end{proof}

\end{document}